\documentclass[a4paper,12pt]{article}
\usepackage{amsmath, amsthm}
\usepackage{amsfonts}
\usepackage{amssymb}
\usepackage[utf8]{inputenc}
\usepackage[T1]{fontenc}
\usepackage{slashed}
\usepackage{commath}
\usepackage{graphicx}
\usepackage{floatrow}
\usepackage[affil-it]{authblk}
\usepackage[a4paper, width=5.7in]{geometry}
\usepackage{setspace}
\usepackage{hyperref}
\usepackage{array}

\renewcommand{\emptyset}{\varnothing}

\newcommand{\GA}[2][]{\ensuremath{\mathcal{G}_{#1}({#2})}}
\newcommand{\grade}[2]{\ensuremath{\langle#2\rangle_{#1}}}

\newcommand{\reverse}[1]{\tilde{#1}} 
\newcommand{\reals}[1]{\ensuremath{\mathbb{R}^{#1}}}
\newcommand{\vol}[1]{\ensuremath{\mathrm{vol}(#1)}}
\newcommand{\linner}{\mathbin{\ensuremath{\rfloor}}}
\newcommand{\rinner}{\mathbin{\ensuremath{\lfloor}}}

\newtheorem{thm}{Theorem}
\newtheorem{lem}[thm]{Lemma}
\newtheorem{cor}[thm]{Corollary}

\theoremstyle{definition}
\newtheorem{defn}{Definition}

\newcommand{\comment}[1]{}

\providecommand{\abs}[1]{\lvert#1\rvert}
\providecommand{\norm}[1]{\left\lVert#1\right\rVert}

\title{Coordinate free integrals in Geometric Calculus}

\begin{document}
	
\author{Timo Alho\\ e-mail: \url{alho@hi.is}}
\affil{Science Institute, University of Iceland\\
	 Dunhaga 5, 107 Reykjavik, Iceland}
\maketitle


\begin{abstract}
	We introduce a method for evaluating integrals in geometric calculus without introducing coordinates, based on using the fundamental theorem of calculus repeatedly and cutting the resulting manifolds so as to create a boundary and allow for the existence of an antiderivative at each step.
	The method is a direct generalization of the usual method of integration on \reals{}. It may lead to both practical applications and help unveil new connections to various fields of mathematics.
\end{abstract}

\newpage

\section{Introduction}

One of the main selling points for Geometric Algebra and Calculus \cite{Hestenes03oerstedmedal, Hestenes03spacetimephysics, GANF, GAMath, doran2007geometric, Gull:1993aa, Lasenby:1993ya, Doran:1993pja, Lasenby:1998yq, UnifiedLanguage} is the claim that it allows carrying out computations in inner product spaces without resorting to coordinates.
Indeed, there exist well developed methods for simplifying algebraic statements and solving equations, computing the vector derivative and the multivector derivative, and finally for developing a theory of directed integration, all in a coordinate free manner.
However, when it comes to actually computing the value of an integral, a coordinate system is invariably introduced \cite{doran2007geometric, GAMath, macdonald2012vector}.
This paper takes key steps towards remedying this.

In calculus on $\mathbb{R}$, definite integration is usually carried out by finding an antiderivative or an indefinite integral of the function to be integrated, and then applying the fundamental theorem of calculus to obtain the desired definite integral.
The fundamental theorem of geometric calculus \cite{GAMath, FundamentalTheorem}, a version of which can be expressed as
\begin{equation}
\int_M \mathrm{d}^m x\ \partial_{M} F = \int_{\partial M} \mathrm{d}^{m-1}x\ F,\label{eq:fund}
\end{equation} 
where $\partial_M$ is the vector derivative on the manifold $M$, provides a tool to do the same in any number of dimensions, for functions with values in the geometric algebra.

Let us briefly recall the main elements in \eqref{eq:fund}. In a directed integral, the integration measure $\dif^m x$ is an $m$-vector valued element of the tangent algebra of $M$, analogous to the volume form in the theory of differential forms. When $M$ is embedded in a higher dimensional manifold, the directed integral therefore carries more information than the usual integral with a scalar valued measure, including information about the orientation of the manifold that the integral is over, weighted by the integrand.

The vector derivative on a manifold, $\partial_M$, is a vector-valued derivative operator, and so in addition to taking derivatives it acts algebraically as a vector. On a manifold, it only considers differences along the manifold, but note that the result of the derivation can take values in the full geometric algebra, so it is distinct from the covariant derivative. In coordinates, one can define $\partial_M = \sum_i p_{M_x}(e^i) \frac{\partial}{\partial x^i}$, where $p_{M_x}(a)$ is the projection of the vector $a$ to the tangent space of the manifold at point $x$. In what follows, we usually suppress $x$ in the notation, and also $M$ where the manifold is clear from context.

Let us present a summary of the method we are proposing, to be elaborated on in the rest of the paper:
assume we are integrating a function $f(x)$ over a $d$-dimensional subset $M$ of $\mathbb{R}^d$, which is sufficiently smooth to satisfy the assumptions of the fundamental theorem and has a finite number of connected components.
The first step is to find an antiderivative  $F_1(x)$ of $f(x)$, i.e. $\partial_M F_1(x) = f(x)$ for all $x \in M$.
Now we get, according to \eqref{eq:fund}, an integral over the $d-1$ dimensional boundary $\partial M$ of $M$.
We'd like to use the fundamental theorem again,  and so we look for an antiderivative $F_2(x)$ of $F_1(x)$ on the boundary $\partial M$ with respect to the derivative $\partial_{\partial M}$ on the boundary.
Given an antiderivative $F_2(x)$ we run into the problem that the boundary of the boundary of a set is always empty.
We move forward by making an incision of the boundary, i.e. we choose a set $E_2$ such that $\partial M \setminus E_2$ has a smooth boundary  $\partial_{E_2} \partial M := \partial (\partial M \setminus E_2)$, and $\mathrm{vol}(E_2) < \epsilon_2$.
Now the integral
\begin{equation}
 \int_{\partial_{E_2}\partial M} d^{d-2}x F_2(x)
\end{equation}
differs from our desired integral by at most $\mathrm{vol}(E_2) \sup_{x \in E_2}\norm{F_2(x)}$.
Notice that we have to choose $F_2$ and $E_2$ such that $F_2$ is continuous in $\partial M \setminus E_2$, in order to justify our use of the fundamental theorem.
This requirement is actually crucial, since any finite value of the integral as we shrink $\epsilon_2$ to zero comes from what are essentially branch cut discontinuities in the antiderivative.
Indeed, due to the presence of branch cuts, we could not have found $F_2(x)$ on the whole manifold, giving a second reason why the incision is necessary.

We then simply repeat the same construction $d$ times, at each step requiring that for incision $E_n$ the volume $\mathrm{vol}(E_n) < \epsilon_n$ and that each antiderivative is continuous in the integration set.
In the final step, the integration will be over a one-dimensional manifold, which simply has a finite number of points as a boundary, leaving us with a finite sum of values of the $d$th antiderivative. Then as we let all of the $\epsilon_n$'s go to zero, we get our final result. 

As will be shown via examples, this method allows computing integrals without invoking a coordinate system.
However, we will find in all practical examples that we do need to invoke reference vectors or multivectors, and the expectation is indeed that this will turn out to be generic, as the reference multivectors provide a mechanism for choosing a specific antiderivative. 

We expect that this method of integration will open up new possibilities in analyzing any integral or differential systems in $n$-dimensions.
This includes the theory of partial differential equations\footnote{
	When such equations are expressed in geometric calculus, we follow \cite{GAMath} in considering this a misnomer, and prefer the term vector differential equation.}, 
numerical estimation methods for integrals, and also connections to algebraic geometry, since it becomes possible, at least in principle, to handle all aspects of surfaces expressible as algebraic equations in a coordinate independent manner.

In this paper, we first prove that when the requisite antiderivatives and submanifolds exist and satisfy a number of reasonable properties, the above construction indeed gives the desired result.
We then give some examples of elementary integrals worked out according to the method.
Finally, we elaborate on possible implications and directions for further research.

\section{Integration by antiderivatives}

Let us briefly recall some definitions and establish some notation.

Our basic notation follows that used by \cite{Chisolm:2012aa}. We use the left- and right contractions $\rinner$ and $\linner$ instead of the single dot product, our scalar product contains the reverse, $A \ast B = \grade{}{A \reverse{B}}$, and our dual is a right multiplication by the pseudoscalar.
The norm on a geometric algebra is defined as $\norm{A}^2 = A\ast A = \grade{}{A \reverse{A}}$.
Although our method generalizes easily to the case of mixed signatures, we will for simplicity consider here only spaces where the inner product is positive definite, and so the multivector norm defines a well-behaved concept of convergence.

Since the directed integral of a multivector function can always be expanded in a multivector basis in terms of scalar coefficient functions, we can import the concept of integrability from scalar valued integrals:
\begin{defn}
	A function $f: M \to \GA[M]{x}$ is L-integrable in the sense of the directed integral on $M$ if each of the scalar functions $a^I(x) \ast (\frac{\dif^m x}{\norm{\dif^m x}} f(x))$ are L-integrable on $M$ with the measure $\norm{\dif^m x}$ for all $a^I$, where $I$ is a multi-index and the set $\{a^I(x)\}$ forms a multivector basis \cite{GAMath, Chisolm:2012aa} of $\GA[M]{x}$, and L is a definition of integrability for scalar valued functions, such as Riemann or Lebesgue.
\end{defn}
In what follows, we will simply refer to integrability, and by that mean integrability in the sense of the directed integral based on a suitable definition of scalar integrability. For all the theorems and examples in this paper, the Riemann integral will be sufficient.

We write $\vol{M}$ for the volume of a manifold in the appropriate dimension, i.e. for $\dim(M) = 2$ the volume is the area, and so on.

For completeness, let us recall the definition of the tangent algebra and the vector derivative \cite{GAMath}:
\begin{defn}
	Let $M$ be a Euclidean vector manifold \cite{GAMath}. Then the \emph{tangent algebra} of $M$ at $x \in M$, denoted by \GA[M]{x} is the geometric algebra, i.e. real Clifford algebra, generated by the tangent space $T_x M$.
\end{defn}

\begin{defn}
Given a vector derivative $\partial_M$ on an orientable vector manifold $M$ and an orientable submanifold $N \subseteq M$ and a unit pseudoscalar of $N$, $I_N(x) \in \GA[M]{x}$, for each $x \in N$, the projected derivative $\partial_N$ is given by \cite{GAMath}
\begin{equation}
 \partial_N = p_{N_x}(\partial) = \sum_i p_{N_x}(e_i) e_i\cdot \partial_M = \sum_i p_{N_x}(e_i)  \frac{\partial }{\partial x_i},
\end{equation}
where $p_{N_x}(a) = I_N(x)^{-1}(I_N(x) \rinner a)$ is the projection of a vector $a$ to the tangent algebra of the manifold $N$ at $x \in N$, and $\{e_i\}$ is a basis of the tangent space $T_x M$.
\end{defn}
Note that the partial derivative operator does not operate on the pseudoscalar $I_N(x)$, and also that the projected derivative can take values in the full tangent algebra of $M$, not just $N$.
In addition, the projections on the basis vectors $e_i$ can be dropped if we let the sum run only over a basis of $T_x N$.
Then one version of the fundamental theorem of calculus can be expressed as \cite{GAMath, FundamentalTheorem}
\begin{thm}[Fundamental theorem of calculus]
	Let $M$ be an oriented $m$-dimensional vector manifold with pseudoscalar $I_M(x)$ and a boundary $\partial M$ that is a vector manifold, $f$ a differentiable function $f: M \to \GA[M]{x}$, and $\partial_M$ the vector derivative on $M$. Then
\begin{equation}
 \int_M \mathrm{d}^m x\ \partial_{M} f(x) = \int_{\partial M} \mathrm{d}^{m-1}x\ f(x),
\end{equation}
 where the pseudoscalar $\mathrm{d}^{m-1}x$ is oriented such that $I_M(x) \norm{\mathrm{d}^{m-1}x} = \mathrm{d}^{m-1}x\ n(x)$, where $n(x)$ is the outward directed unit normal of $\partial M$ at $x$.
\end{thm}

Note that since the measure is pseudoscalar-valued, its position relative to the integrand matters.
Indeed, the most general form of the theorem is concerned with integrals of the form $\int L(\mathrm{d}^m x)$, where $L(\mathrm{d}^m x)$ is a linear function of $\mathrm{d}^m x$ \cite{GAMath}.
However, we will only consider the form with the pseudoscalar measure to the left of the integrand in this paper. 
Let us point out some consequences of the requirement concerning the orientation of the pseudoscalar of $\partial M$:
\begin{itemize}
	\item when we make a very small incision on a manifold, the pseudoscalars of the newly created boundary at two nearby points $x_1$ and $x_2$, on opposite sides of the incision, will be related by $I(x_1) \approx -I(x_2)$, since the corresponding outward normals will be nearly opposite.
	This is what guarantees, at the level of the fundamental theorem, that small incisions in a region where the antiderivative is continuous have a small effect on the value of the integral.
	\item when $M$ is a 1-dimensional manifold, so that its pseudoscalar $\mathrm{d}x$ is a vector, then the unit pseudoscalar of $\partial M$, $\mathrm{d}^0 x$, is the scalar $\pm 1$.
	Indeed, one can think of $\mathrm{d}^0x$ as a signed counting measure.
	As the direction of $\mathrm{d}x$ is continuous over the curve, this sign will specifically be $+1$ at one of the endpoints and $-1$ at the other, in accordance with the fundamental theorem of calculus on \reals{}.
\end{itemize}

Let us then prove a simple lemma:
\begin{lem}\label{lem:inclusion}
	Let $M$ be an oriented $m$-dimensional vector manifold and $f:M \to \GA[M]{x}$ be an  integrable function from the manifold to the algebra. Given a bounded submanifold $E \subset M$ such that $f$ is bounded in $E$, then
	\begin{equation}
	 \norm{\int_{M\setminus E} \dif^mx\ f(x) - \int_{M} \dif^mx\ f(x)} \leq \vol{E} \sup_{x\in E}\norm{f(x)}
	\end{equation}
\end{lem}
\begin{proof}
	Direct calculation using the triangle inequality:
	\begin{equation}
	\begin{split}
	&\norm{\int_{M\setminus E} \dif^mx\ f(x) - \int_{M} \dif^mx\ f(x)} = \norm{\int_{E} \dif^mx\ f(x)} \leq \int_{E} \norm{\dif^mx\ f(x)}\\
	 &=  \int_{E} \norm{\dif^mx}\norm{ f(x)} \leq \int_{E} \norm{\dif^mx}\sup_{x \in E}\norm{f(x)} = \vol{E} \sup_{x\in E}\norm{f(x)}.
	 \end{split}
	\end{equation}
	Note that the suprema exists and is finite since $E$ is bounded and $f$ is bounded on $E$.
\end{proof}

\comment{
\begin{lem}\label{lem:incision}
	Let $M$ be an $m$-dimensional manifold and $E\subset M$ be as in Lemma~\ref{lem:inclusion}, $f: M \to \GA[M]{x}$ a bounded function with the antiderivative $F: M \setminus E \to \GA[M]{x}$ on $M \setminus E$, $\partial_p F(x) = f(x)$ s.t. the conditions for the fundamental theorem of calculus are satisfied. Then
	\begin{equation}
	\norm{\int_{M} \dif^mx\ f(x) - \int_{\partial (M \setminus E)} \dif^{m-1}x\ F(x)} \leq \epsilon,
	\end{equation}
	where $\epsilon = \delta \sup_{x\in E}\norm{f(x)}$.
\end{lem}
\begin{proof}
	Using the fundamental theorem and Lemma~\ref{lem:inclusion}:
	\begin{eqnarray}
	&&\norm{\int_{M} \dif^mx\ f(x) - \int_{\partial (M \setminus E)} \dif^{m-1}x\ F(x)}  \nonumber\\ 
    &=&\norm{\int_{M} \dif^mx\ f(x) - \int_{M \setminus E} \dif^{m}x\ f(x)}
	\leq  \delta \sup_{x\in E}\norm{f(x)}.
	\end{eqnarray}
\end{proof}
}
The point of Lemma~\ref{lem:inclusion} is that it allows us to cut out a part of the manifold in order to guarantee that it has a boundary, and still keep control of the error we are making.
Also, we will find out that usually functions on manifolds without boundary do not have single valued antiderivatives, and the lemma allows us to exclude a branch cut, since the existence of the antiderivative is only necessary on the part of the manifold that is not cut.

\begin{defn}
	Let $M$ be a vector manifold and $f: M \to \GA[M]{x}$ be a function on the manifold.
	If $f$ has an antiderivative $F$ on $M$, we write $F =: \partial^{-1}_M f$.
	If $\partial^{-1}_M f$ again has an antiderivative on $N \subseteq M$, we denote that by $\partial^{-2}_{M N}f$, and in general we write $\partial^{-n}_{M_1 M_2 \ldots M_n}f$ for the $n$th antiderivative of $f$ on the manifold $M_n$, if it exists, with $M_1 \subseteq M_2 \subseteq \ldots \subseteq M_n$.
	\label{def:antiderivative}
\end{defn}
Note that due to the projection operator in the derivative on a manifold, the antiderivative in general depends on the manifold in which it is defined. In other words an antiderivative on a submanifold is not necessarily just the restriction of some antiderivative on the full manifold. Also, in the above definition the antiderivative is ambiguous, so when using the notation we have to either define how to choose a specific antiderivative, or show that our results don't depend on the choice.

Now we get to the main result:
\begin{thm}\label{thm:main}
	Let $M$ be an $m$-dimensional orientable vector manifold, and $f: M \to \GA[M]{x}$ an integrable function. If there exists a sequence of orientable manifolds $N_0 \subset N_1 \subset \ldots \subset N_{m} = M$ and a sequence of bounded sets $E_i$ such that
	\begin{itemize}
		\item if $\partial N_{i+1} \neq \emptyset$, $N_i = \partial N_{i+1}$, otherwise $N_i = \partial(N_{i+1}\setminus E_{i+1})$, where $E_{i+1}$ is a bounded set such that the boundary $\partial (N_{i+1} \setminus E_{i+1})$ is a non-empty vector manifold, and $\partial^{-m + i + 1}_{N_{i+1}\ldots N_m} f$ is integrable and bounded on $E_{i+1}$.
		
		\item there exists an antiderivative $\partial^{-m+i}_{N_{i} \ldots N_m} f$ on $N_i$, which is bounded.
		
		\item $N_0$ is a finite set
	\end{itemize}
	then the integral of $f$ over $M$ can be computed by evaluating the $m$th antiderivative on $N_0$:
		 \begin{equation}
				\norm{\int_M \dif^mx\ f(x) - \sum_{x_i\in N_0} \partial^{-m}_{N_0 \ldots N_m}s_i f(x_i)} \leq \epsilon,
			  \end{equation}
			  where $\epsilon = \max_i \sup_{x \in E_i}\norm{\partial^{-m+i}_{N_{i} \ldots N_m}f(x)} \sum_i \vol{E_i}$, and the signs $s_i \in \{-1, 1\}$ are determined by fulfilling the requirement on the orientation of the boundary in the fundamental theorem at each step.
\end{thm}
Before proving the theorem, we make a few remarks.
We basically forced the theorem to be true by sticking all the difficult parts into the assumptions.
Note however that the local existence of an antiderivative is guaranteed for a differentiable function \cite{macdonald2012vector, macdonaldGC, FundamentalTheorem}, and also that the set $N_0$ is automatically discrete since it is the boundary of a 1-dimensional manifold, and with very mild assumptions on $M$ the $E_i$ can be chosen such that $N_0$ is a finite set.
In essence these assumptions allows us to prove the theorem without getting mixed up in topological complications, and for most practical applications the natural choice of the sets $N_i$ will anyway fulfill these assumptions, which is why we are not interested in sharpening the theorem at this point.\footnote{
	Since in many applications there may be a branch cut that goes to infinity, relaxing the assumption about $E_i$'s being bounded would be beneficial, allowing to compute also such integrals when they are finite. 
	This would entail finding a sufficient set of assumptions to guarantee that $\int_{E_i} \partial^{-m + i}_{N_i\ldots N_m} f$ goes to zero as the set $E_i$ shrinks to zero. 
	In specific cases this should not be difficult.
	}

\begin{proof}[Proof of theorem~\ref{thm:main}]
	First note that since the integral of a bounded function over a bounded set is finite, each of the suprema in the expression for $\epsilon$ exist.
	The only part left to prove is the inequality.
	Using the fundamental theorem, Lemma~\ref{lem:inclusion} and the triangle inequality, we first compute
	\begin{equation}
	\begin{split}
		&\norm{\int_M \dif^mx\ f(x) - \sum_{x_i\in N_0} \partial^{-m}_{N_0 \ldots N_m}s_i f(x_i)}\\
		 =& 	\norm{\int_M \dif^mx\ f(x) - \int_{N_1 \setminus E_1} \dif x \partial^{-m + 1}_{N_1 \ldots N_m} f(x)}\\
		 =& 	\left\| \int_M \dif^mx\ f(x) - \int_{N_1\setminus E_1} \dif x \partial^{-m + 1}_{N_1 \ldots N_m} f(x) \right.\\
		  &+\left. \int_{N_1} \dif x \partial^{-m + 1}_{N_1 \ldots N_m} f(x) -
		 	 \int_{N_1} \dif x \partial^{-m + 1}_{N_1 \ldots N_m} f(x)\right\|\\
		 	 \leq& \norm{\int_M \dif^mx\ f(x) - \int_{N_1} \dif x \partial^{-m + 1}_{N_1 \ldots N_m} f(x)}\\
		 	 &+ \left\| \int_{N_1\setminus E_1} \dif x \partial^{-m + 1}_{N_1 \ldots N_m} f(x) -
		 	 \int_{N_1} \dif x \partial^{-m + 1}_{N_1 \ldots N_m} f(x)\right\|\\
		 	 \leq& \norm{\int_M \dif^mx\ f(x) - \int_{N_1} \dif x \partial^{-m + 1}_{N_1 \ldots N_m} f(x)}\\
		 	 &+ \vol{E_1} \sup_{x\in E_1} \norm{\partial^{-m+1}_{N_1\ldots N_m} f(x)}.
	\end{split}		 	 
	\end{equation}
	
	Note that since two antiderivatives differ at most by a monogenic function $\psi$ for which $\partial_{N_0} \psi(x) = 0$ \cite{GAMath}, this result is independent of the choice of antiderivative, resolving the caveat mentioned in definition \ref{def:antiderivative}.
	Also, the signs $s_i$ must indeed follow the orientation requirement of the fundamental theorem to allow representing the sum as an integral.
	
	We can then continue using similar steps, each of which produces an approximation error $\vol{E_i}\sup_{x\in E_i}\norm{\partial^{-m+i}_{N_i\ldots N_m}}$, until finally at the $m$'th step, we get
	\begin{equation}
	\norm{\int_M \dif^mx\ f(x) - \int_{N_m} \dif^mx\ \partial^{0}_{N_m}f(x)} + \sum_i  \vol{E_i} \sup_{x\in E_i} \norm{\partial^{-m+i}_{N_i\ldots N_m} f(x)},
	\end{equation}
	where $N_m = M$ and $\partial^0_{M} f(x)$ is the function itself, and so the integral term is zero.
	Approximating the suprema by their maximum concludes the proof.
\end{proof}

There is a simple corollary;
\begin{cor}
	Let $M$ be a vector manifold without a boundary, and $f: M \to \GA[M]{x}$ be a bounded integrable function such that its integral over $M$ is non-zero.
	Then any antiderivative $\partial^{-1}_M f$ of $f$ must have a branch cut discontinuity which divides the manifold into at least two parts with non-zero volumes.
\end{cor}
\begin{proof}
	Assume the opposite, that is, that there exists an antiderivative of $f$ on the whole of $M$.
	Then we can make a cut according to theorem~\ref{thm:main}, and let its volume shrink to zero.
	Since the antiderivative of a bounded function is bounded (which can be seen, for example, by considering the scalar components and applying the usual theorems of integration), this means that the result of the integration is zero.
	This is a contradiction.
\end{proof}
In particular, this means that the norm of the volume form on a manifold without boundary cannot have an antiderivative everywhere.
Also, since every function on a manifold is an antiderivative of its own derivative, this corollary may have some links to the hairy ball theorem.

Note also that even though the method is phrased in terms of the directed integral, it is immediately applicable to the usual integral with a scalar measure. We simply write $\norm{\dif^m x} f(x) = \dif^m x I(x) f(x)$, where $I(x)$ is the unit pseudoscalar of the manifold at $x$.

In order to do a specific calculation, we find the necessary antiderivatives and sets to cut out by any means we like, and then using theorem~\ref{thm:main}, we can be assured that as we let the volume of the incisions $E_i$ go to zero we get the exact value of the integral.
Note that since the errors are additive, the order of the limits for the various sets does not matter (unless their construction dictates a specific order).
Of course, we have only proven that if this construction can be made, then we can do the coordinate free integral.
Let us next present some examples to show that such constructions indeed do exist.

\section{Examples}

Next we compute examples of applying this method of integration.
Since these quite trivial examples already show many of the features we expect to encounter in more generic cases, we work them out in detail.
The algebra and rules for computing the derivatives needed in this section are contained, for example, in \cite{GAMath, doran2007geometric, HitzerCalculus, Chisolm:2012aa}.

\subsection{The area of a disk}

As the first example of application of the method, we calculate the area of a disk of radius $r$ in $\mathbb{R}^2$.
The integral we intend to compute is
\begin{equation}
A_{B_r} = \int_{B_r} \dif\, ^2 x,
\end{equation}
where $B_r = {\{x\in \reals{2}:\ \norm{x} < r\}}$. Note that since the directed volume element $\dif\, ^2 x$ is a bivector, we expect to get the result as a bivector.
We define the corresponding unit bivector $I_2 = \frac{\dif\, ^2 x}{\norm{\dif\, ^2 x}}$.
The first step is to find the antiderivative of the constant function $1$.
This is by inspection $\frac{1}{2}x$, since in general the derivative $\partial_M x$ is $m$, where $m$ is the dimension of the manifold \cite{GAMath, HitzerCalculus}.
Therefore, the integral is reduced to
\begin{equation}
\label{eq:circleintegral}
\frac{1}{2} \int_{S^1} \dif x\, x,
\end{equation}
where $\dif x$ is the vector-valued measure on the circle.
Now the projection of a vector $a$ to $S^1$ at point $x$ is $p_{S^1}(a) = x^{-1}(x \wedge a)$.
Intuitively, we see that the integral to calculate measures distance along the circle, i.e. the angle.
So does the complex logarithm, and so we are led to the try the function $\log(x x_0)$, where $x_0$ is an arbitrary constant vector in \GA{\reals{2}}, and since $x x_0$ is in the even subalgebra of \GA{\reals{2}} which is isomorphic to the complex numbers with the unit pseudoscalar $\frac{x\wedge x_0}{\norm{x \wedge x_0}} = -I_2$ acting as the imaginary unit, the logarithm may be defined analogously to the complex logarithm.
The negative sign appears when comparing the orientation of $\dif\, ^2 x$ to that of $x\wedge x_0$ via the requirement $\dif x \,\hat{x} = \norm{\dif x} \dif\, ^2 x$, coming from the fundamental theorem, where $\hat{x}$ is the unit normal at $x$, and choosing the positive sense of rotation to be counterclockwise.

In order to compute the projected derivative, we observe that in general $\partial_M f(x) = \dot{\partial}_M (\dot{a}\linner \partial_x) f(x)$, where $\partial_x$ is the full vector derivative without the projection, and the overdot denotes that the derivative $\dot{\partial}_M$ acts only on $a$.
Then, using the chain rule and the fact that the derivative $(x x_0)\ast \partial_z$ reduces to the directed derivative in the direction $x x_0$ \cite{HitzerCalculus}, which further reduces to the complex derivative times $x x_0$ since the direction commutes with the argument, we can further calculate 
\begin{equation}
\begin{split}
\partial_{S^1} \log(x x_0) &= \dot{\partial}_{S^1} (\dot{x}x_0)\ast \partial_z \log{z}|_{z = x x_0} = \dot{\partial}_{S^1} (\dot{x} x_0) z^{-1}|_{z = x x_0}\\
&= x_0 \frac{x_0 x}{\norm{x x_0}^2} = x^{-1},
\end{split}
\end{equation}
where the overdot limits the scope of the derivative to the dotted objects, as in \cite{GAMath}.
We observe that $\partial_{S^1} x^2 = 0$, as expected, and therefore deduce immediately that $\partial_{S^1} \frac{1}{2}x^2 \log(x x_0) = \frac{1}{2}x$, which is our antiderivative.
The boundary of $S^1$ is empty, but according to our method we cut a small segment, for example the part where $\frac{\abs{x\cdot x_0}}{\norm{x x_0}} > \cos \epsilon$ which is the part at an angle less than $\epsilon$ to $x_0$.
The complex logarithm function is bounded away from zero, and our incision is bounded, so the assumptions of theorem~\ref{thm:main} are satisfied and we calculate
\begin{equation}
\label{eq:logresult}
\int_{S^1} \dif x \frac{1}{2} x = \sum_{x_i \in \partial(S^1 \setminus \{x :\ \norm{x - x_0} < \epsilon\})} s_i \frac{1}{2} x_i^2 \log(x_i x_0).
\end{equation}
Let us choose the branch of the complex logarithm such that $\log(xx_0)|_{x = x_0} = \log\norm{x x_0} - 0 I_2$.
We observe that since the antiderivative must be continuous inside the set where we made the cut, we must then allow the logarithm to approach the value $\log \norm{x x_0} - 2 \pi I_2$ on the other side of the cut, where the negative sign comes from the sign difference between $I_2$ and $x\wedge x_0$.
Note that this puts the branch cut on the positive real axis on the complex plane spanned by $1$ and $I_2$.
The signs $s_i$ are fixed by the fundamental theorem: at the beginning of the interval, $\dif x$ points to the outside of the region, so the "pseudoscalar" must be $s_0 = 1$ to keep the outward unit normal in the same direction.
At the end of the interval $\dif x$ points in the inward direction, and we get $s_1 = -1$.
Therefore the sum results in $A_{B_r} = \pi r^2 I_2$, as expected.

\subsection{The volume of a cylinder}

Let us do an example in three dimensions.
Let $M$ be the cylinder defined by the equations
\begin{align}
 I_3 \wedge x &= 0\label{eq:cylind0}\\
(\omega \rinner x)^2 &\leq r^2 \label{eq:cylind1}\\
 0 \leq (\omega^{-1} (\omega \wedge x)) \rinner (\omega I_3)&\leq h\label{eq:cylind2},
\end{align}
where $\omega$ is a unit bivector determining the plane orthogonal to the axis of the cylinder, $r$ and $h$ are positive real numbers, and $I_3$ is the pseudoscalar of the 3D space in which the cylinder lies.
Eq.~\eqref{eq:cylind0} guarantees that the cylinder is in the space determined by $I_3$ and effectively reduces the problem to three dimensions, whereas Eq.~\eqref{eq:cylind1} sets the radius of the cylinder.
Eq.~\eqref{eq:cylind2} sets the height of the cylinder.

\begin{figure}[h]\centering
\fcapside[\FBwidth]
		{\caption{The cylinder to be integrated. 
				The red translucent part is the chamfer which we cut away before the first integration.
				Note that while its surface does not have a pseudoscalar defined everywhere, the volume itself does.
				The bottom and top of the cylinder are in the plane defined by the bivector $\omega$.}\label{fig:cylinder}}
{\includegraphics[trim=40 60 40 20, clip, width=0.28\textwidth]{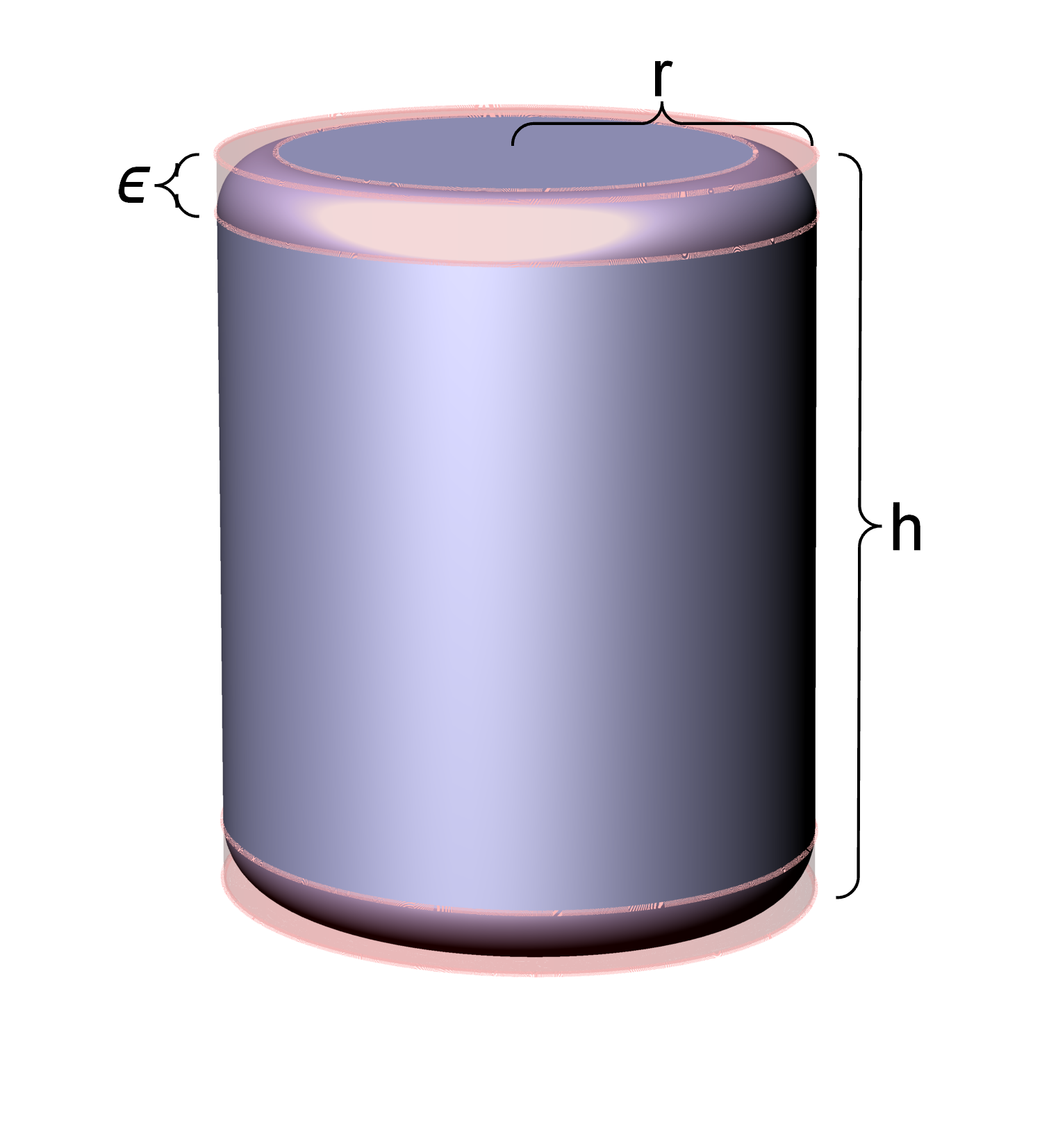}}
\end{figure}

In this case, the cylinder has a sharp edge, which would, after the first integration, contradict the assumption that the pseudoscalar of the surface exists everywhere.
Let us therefore this time use lemma~\ref{lem:inclusion} to cut a circular chamfer of radius $\epsilon$ to the edges, as depicted in \figref{fig:cylinder}, such that the remaining manifold is smooth.
The chamfer has a volume proportional to $\epsilon^2$.
Note that in all three parts the 3D pseudoscalar is well defined everywhere.

The first integral is again trivial, resulting in $\frac{1}{3} x$, since the cylinder is lying in a flat three-dimensional space, and we are integrating the constant function.
After this, we again use lemma~\ref{lem:inclusion} to ignore the surface of the chamfer, and only concern ourselves with the flat parts of the surface integral.
For the surface integral along the sides, we first observe that, with $f(x) = (\omega \rinner x)^2$ being the function whose constant value surface $f(x) = r^2$ defines the side of the cylinder, and given a point $x$ on the side, the projection of a vector $a$ to the tangent space is given by
\begin{equation}
 p_\mathrm{side}(a) = (\partial f(x) I_3)^{-1} (\partial f(x)I_3)\rinner a = r_{\omega}(a) + p_{x\linner \omega}(a),\label{eq:tangentsplit}
\end{equation}
where
\begin{equation}
 r_{\omega}(a) = \omega^{-1}(\omega \wedge x)\quad \text{and\ }\\
 p_{x\linner \omega}(a) = (x\linner \omega)^{-1}(x\linner \omega) \rinner a
\end{equation}
are the rejection from, i.e. part orthogonal to, $\omega$, and the projection to the direction of the vector $x \linner \omega$, which lies in the plane of $\omega$ and orthogonal to $x$, respectively.

We find the antiderivative $\partial^{-1}_{\textrm{side}} x =  x\, r_{\omega}(x)$.
This can be verified by taking the derivative and using the facts that $x = p_{\omega}(x) + r_{\omega}(x)$, where $p_{\omega}(x)$ is the projection to $\omega$, the fact that since the projection to the tangent space splits as in Eq.~\eqref{eq:tangentsplit} then also the derivatives split in the same way, and finally that $\partial_p x\wedge a = d_p a - p(a)$, where $\partial_p$ is the derivative projected with the projection $p$ and $d_p$ is the dimension of the subspace projected to.

In order to do the final integral for the side along the boundary left by the chamfer cut, which is a circle in the plane $\omega$, and at height $h - \epsilon$ above the origin, we note that $r_{\omega}(x)$ is simply the constant vector height along the circle and therefore also constant with respect to the derivative on that circle, so we are left with integrating $x = p_{\omega}(x) + r_{\omega}(x)$ on the circle.
Now $p_{\omega}(x)$ is on the plane of the circle, and therefore we know from the disk example that the integral of the $p_{\omega}(x)$ -part will be $2\pi\norm{p_{\omega}(x)}^2 I_2$ with $I_2 = \omega$ and $\norm{p_\omega(x)}^2 = r^2$.
Integrating the constant produces $x$ times the constant, and since $x$ is regular on the whole circle, the subtraction will produce 0.
The other boundary component is the circle along the bottom, where the calculation is identical expect that now $\norm{r_{\omega}(x)} = \epsilon$, and the sign is opposite since the orientation of the boundary is opposite.
The integral along the sides then total $\frac{2\pi}{3} r^2 (h-2 \epsilon) I_3$, where the pseudoscalar $I_3$ comes from the product of the bivector $\omega$ and vector $r_\omega(x)$.

The other boundary components are the caps on the top and the bottom.
The projection to the tangent plane is simply $p_\omega$, and therefore splitting again $x = p_{\omega}(x) + r_{\omega}(x)$, we find the antiderivative
\begin{equation}
 \partial^{-1}_{\omega} x = \frac{1}{2} p_{\omega}(x)^2 + \frac{1}{2} p_{\omega}(x) r_{\omega}(x).
\end{equation}
The half on the second term comes from the fact that the projection is two-dimensional.
We have to integrate this on the boundary of the cap, which is the circle at radius $r - \epsilon$ (since we cut the chamfer off the edge).
The first term again integrates to zero, since on the circle $(p_\omega(x))^2$ is a constant, whereas the second term again reduces to the case of the disk, and therefore produces $\frac{\pi}{3} (r-\epsilon)^2 \omega r_{\omega}(x)$, where we have inserted the $1/3$ from the first integral.
The cap on the bottom is again the same, with this time $\norm{r_{\omega}(x)} = \epsilon$, and so putting the caps and the side together and letting $\epsilon \rightarrow 0$ we get the final result
\begin{equation}
 \int_{\textrm{cylinder}} \dif\,^3\,x = \pi r^2 h I_3
\end{equation}
as expected.

\section{Toward a more systematic method}

The above examples are calculated rather ad hoc, in the sense that the antiderivatives are guessed and then checked by derivation.
The path toward a more systematic method for calculating coordinate-free integrals is however clear: first, a systematic table of antiderivatives needs to be built by reading tables of vector derivatives in inverse.
To provide an example, table \ref{tab:antiderivatives} lists some such antiderivatives.
\begin{table}
	\begin{tabular}{ >{$\displaystyle}c <{$}| >{\centering$\displaystyle}m{70mm} <{$} m{30mm}}
		f(x) & \partial_{\reals{d}}^{-1} f(x) & Remarks\\
		\hline
		1 & \tfrac{1}{d} x + C(x) &\\[18pt]
		x & \tfrac{1}{2} x^2 + C(x) & \\[18pt]
		\hat{x} & \abs{x} + C(x) &\\[18pt]
		ax & \frac{2 x(x\linner a) - \frac{1}{2} d x^2 a}{d+2} + C(x) & $a$ is a constant vector\\[18pt]
		f(\norm{x}) & \frac{x}{\norm{x}^{d}} \int^{\norm{x}} s^{d-1} f(s) \dif s + C(x)& $f: \reals{} \to \reals{}$ and $f$ is integrable\\[20pt]
	\end{tabular}
	\caption{Some antiderivatives in flat $d$-dimensional space. Here $C(x)$ is the analogue of the constant of integration, which now becomes an arbitrary monogenic function, i.e. $\partial C(x) = 0$.}
	\label{tab:antiderivatives}
\end{table}
Once such a table exists in flat space, there is no need to generate a new one for each manifold.
Rather, given a function on a manifold with a known embedding in flat space, we can simply use (the inverse of) the embedding function to map the function to a flat subspace of the embedding manifold.
The change of variables induces a mapping of the pseudoscalar via its differential outermorphism \cite{GAMath}, from which we extract the pseudoscalar of the flat space.
The product of the part extracted from the pseudoscalar and the function mapped to the flat subspace can then be integrated using the table of antiderivatives in flat space.

For example, when integrating a function on the circle, the mapping
\begin{equation}
 y(x) = \log(x x_0)x_0
\end{equation}
takes $x$ to a vector $y$, with a constant length in the direction of $x_0$, and changes in $x$ along the circle affect $y$ only in a direction orthogonal to $x_0$.
In other words, the circle is mapped to a line.
The pseudoscalar $\mathrm{d}x$ on the circle is derived as
\begin{equation}
\label{eq:circlecov}
\mathrm{d}x = \underline{y}^{-1}(\mathrm{d}y) = \mathrm{d}y\ x_0^{-1}e^{yx_0^{-1}}x_0^{-1},
\end{equation}
where $\underline{y}^{-1}$ is the inverse of the differential of $y(x)$.
This gives a method for reducing an integral on the circle to an integral on a straight line.
Inserting \eqref{eq:circlecov} to \eqref{eq:circleintegral} immediately yields \eqref{eq:logresult}.

The above method is of course closely analogous to an ordinary change of variables in coordinate-based methods of integration, with the Jacobian appearing in the mapping of the pseudoscalar.
It also requires an explicit expression for an embedding in a flat space.
At least for the case of a manifold defined by $m(x) = 0$, where $m(x)$ is a scalar-valued function, we can sketch an alternative method.
The pseudoscalar of the manifold at point $x$ is $(\partial m(x))I$, where $I$ is the unit pseudoscalar of the embedding space.
We then look for a mapping $y(x)$ such that
\begin{equation}
\label{eq:pseudoscalardiff}
 \underline{y}((\partial m(x))I) = I_0,
\end{equation}
where $I_0$ is the constant pseudoscalar of a flat subspace of the embedding space.
Note that $\underline{y}$ has to be linear in its argument, but can depend on $x$ in a complicated way.
Then \eqref{eq:pseudoscalardiff} is a differential equation for the mapping which, once solved for a given manifold, reduces integrals of functions on the manifold to integrals on a flat space.

Finally, as already integrals of functions of a real variable can rarely be evaluated analytically in terms of a finite set of elementary functions, we cannot expect to do any better in this generalized case.
Therefore the ultimate goal must be a coordinate-free approximation theory, which would allow evaluating integrals of sufficiently smooth functions in a similar way as an integral for a real analytic function can always be evaluated in terms of a Taylor series.
We however leave that problem for a later work, although with some speculation about possible properties of such approximations in the next section.

\section{Conclusions and outlook}

We have presented a method for computing integrals in $m$ dimensions without using coordinates.
Naturally, the level of freedom from using coordinates depends on how the manifold and the integrand are defined.
One purely coordinate free way is to define the manifold by solutions of $m(x) = 0$, where $m(x)$ is a function of the vector $x$ constructed from geometric products of $x$ with itself and some (possibly infinite) set of constant multivectors $A_i$, where the geometric relations between $A_i$ and $x$ are known in sufficient detail to allow carrying out all the necessary algebraic manipulations without coordinates.
Both of our examples are of this form.

In the examples, we integrate the constant function on two manifolds in order to compute their volumes. 
The actual computations in these examples are not complicated when compared to the same computation in coordinates, which for a fair comparison needs to take into account the derivation of the Jacobian in polar or cylindrical coordinates.
Further development of our method will indeed require building a comprehensive toolbox of systematic methods for finding antiderivatives of multivector valued functions of vector variables on vector manifolds.
While this program is still in its infancy, we have found some rules with some level of generality: for example, as shown in table \ref{tab:antiderivatives}, an antiderivative of $f(\norm{x})$ in $d$-dimensions is simply $\frac{x}{\norm{x}^{d}} \int \dif s\ s^{d-1} f(s)$, where $f(s)$ is a scalar valued function of a scalar, and so the remaining integral is an ordinary scalar integral. This rule is of course equivalent to integrating in a spherical coordinate system, expressed in a coordinate-free way.

As an interesting note, in some examples which we have worked out but not reported here, such as the volume of $B^3$, it is not necessary to actually find an antiderivative, but rather one can find a function whose derivative differs from the desired one by a function which can be seen to integrate to zero.
We can then use such a function instead of the antiderivative to still get the correct result.
However, we will not comment on this further before we understand the phenomenon in more detail. It may turn out to be only a fortunate coincidence occurring in a limited number of cases, rather than something that can be included in a general toolbox.

Let us indulge in some speculation concerning possible applications of the method to more than just evaluating integrals in the few special cases where antiderivatives can be explicitly found.
Consider a function $f(x)$ on a manifold $M$ defined by $m(x) = 0$ for some multivector valued function $m(x)$ and with $x$ in $\reals{d}$. 
In order to calculate the integral of $f(x)$ over $M$, the method involves finding the $d$-fold antiderivative of $f$ with respect to derivatives projected on $M$, and evaluating it on a discrete set of points on the manifold.
Therefore, at least in the final step, we only really need to know some topological facts about the manifold in order to choose the points such that they are all on the same branch of the antiderivative.
Of course, the manifold also enters into the calculation via the projections of the derivative operator.
For the first integration in the case where $m(x)$ is scalar-valued the projected derivative is given simply by $(\partial m(x) I_d)^{-1} (\partial m(x) I_d) \rinner \partial$, where the first two $\partial$~'s affect only the $m(x)$~'s immediately following them.
Similar formulas can be worked out for more general $m(x)$.
Now, we can use the Taylor series approximation for multivector functions \cite{HitzerCalculus} and approximate both functions $f(x)$ and $m(x)$ by their Taylor series.
If the antiderivatives of all the monomial terms\footnote{
	We need to also expand the inverse appearing in the projection, or to integrate a rational function of multivectors, which cannot be done in quite closed form even for the real numbers, as the roots of the polynomials need to be found in the partial fraction expansion.}
can be explicitly constructed, then this should in principle allow for a systematic series expansion for the values of integrals on a large class of manifolds, in terms of integrals of the monomials.
The theoretical connections to algebraic geometry and topology should prove interesting.

For (vector) differential equations the very same rules for finding antiderivatives that are crucial for our method will be useful in finding closed form solutions in a coordinate invariant way.
In addition, similar series expansion methods as those outlined above should pave the way to finding series expansions for solutions of vector differential equations, and may even aid in their numerical evaluation.

On a philosophical level, our method represents a further step into the direction of establishing multivectors as geometric numbers, which can indeed be constructed, manipulated and interpreted in a wholly coordinate-free way.

\section*{Acknowledgements}

We thank A. Lewandowski and L. Thorlacius for very helpful comments and proofreading of the manuscript. The author is supported in part by Icelandic Research Fund grant
130131-053 and by a grant from the University of Iceland Research Fund.

\bibliographystyle{nb}
\bibliography{GARefs}

\begin{thebibliography}{10}
\ifx\href\asklfhas\newcommand{\href}[2]{#2}\fi
\ifx\arxivref\asklfhas\newcommand{\arxivref}[2]{\href{http://arxiv.org/abs/#1}{#2}}\fi
\ifx\doiref\asklfhas\newcommand{\doiref}[2]{\href{http://dx.doi.org/#1}{#2}}\fi
\raggedright
\small
\parskip 0pt

\bibitem{Hestenes03oerstedmedal}
D.~Hestenes,
\textit{``Oersted Medal Lecture 2002: Reforming the mathematical language of
  physics''},
\textsf{Am.~J.~Phys~71,~104~(2003)},
\href{http://geocalc.clas.asu.edu/pdf/OerstedMedalLecture.pdf}{\texttt{http://geocalc.clas.asu.edu/pdf/OerstedMedalLecture.pdf}}.

\bibitem{Hestenes03spacetimephysics}
D.~Hestenes,
\textit{``Spacetime Physics with Geometric Algebra''},
\textsf{Am.~J.~Phys~71,~691~(2003)},
\href{http://geocalc.clas.asu.edu/pdf/SpacetimePhysics.pdf}{\texttt{http://geocalc.clas.asu.edu/pdf/SpacetimePhysics.pdf}}.

\bibitem{GANF}
D.~Hestenes,
\textit{``New Foundations for Classical Mechanics: Fundamental Theories of
  Physics (Fundamental Theories of Physics)''},
Kluwer Academic Publishers; 2nd edition (1999).

\bibitem{GAMath}
D.~Hestenes and G.~Sobczyk,
\textit{``Clifford Algebra to Geometric Calculus: A Unified Language for
  Mathematics and Physics (Fundamental Theories of Physics)''},
Kluwer Academic Publishers (1987).

\bibitem{doran2007geometric}
C.~Doran and A.~Lasenby,
\textit{``Geometric algebra for physicists''},
Cambridge University Press (2007).

\bibitem{Gull:1993aa}
S.~Gull, A.~Lasenby and Doran,
\textit{``Imaginary Numbers are not Real - the Geometric Algebra of
  Spacetime''},
\textsf{Found.~Phys.~23,~1175~(1993)},
\href{http://www.mrao.cam.ac.uk/~clifford/publications/abstracts/imag\_numbs.html}{\texttt{http://www.mrao.cam.ac.uk/~clifford/publications/abstracts/imag\_numbs.html}}.

\bibitem{Lasenby:1993ya}
A.~Lasenby, C.~Doran and S.~Gull,
\textit{``{A Multivector derivative approach to Lagrangian field theory}''},
\textsf{\doiref{10.1007/BF01883781}{Found.Phys.~23,~1295~(1993)}},
\href{http://www.mrao.cam.ac.uk/~clifford/publications/abstracts/lag\_field.html}{\texttt{http://www.mrao.cam.ac.uk/~clifford/publications/abstracts/lag\_field.html}}.

\bibitem{Doran:1993pja}
C.~Doran, A.~Lasenby and S.~Gull,
\textit{``{States and operators in the spacetime algebra}''},
\textsf{\doiref{10.1007/BF01883678}{Found~Phys.~23,~1239~(1993)}},
\href{http://www.mrao.cam.ac.uk/~clifford/publications/abstracts/states.html}{\texttt{http://www.mrao.cam.ac.uk/~clifford/publications/abstracts/states.html}}.

\bibitem{Lasenby:1998yq}
A.~Lasenby, C.~Doran and S.~Gull,
\textit{``{Gravity, gauge theories and geometric algebra}''},
\textsf{\doiref{10.1098/rsta.1998.0178}{Phil.Trans.Roy.Soc.Lond.~A356,~487~(1998)}},
\texttt{\arxivref{gr-qc/0405033}{gr-qc/0405033}}.

\bibitem{UnifiedLanguage}
J.~Lasenby, A.~N.~Lasenby and C.~J.~L.~Doran,
\textit{``{A unified mathematical language for physics and engineering in the
  21st century}''},
\textsf{\doiref{10.1098/rsta.2000.0517}{Philosophical~Transactions~of~The~Royal~Society~B:~Biological~Sciences~358,~C.~J.~L.~Doran~(2000)}},
\href{http://www.mrao.cam.ac.uk/~clifford/publications/abstracts/dll\_millen.html}{\texttt{http://www.mrao.cam.ac.uk/~clifford/publications/abstracts/dll\_millen.html}}.

\bibitem{macdonald2012vector}
A.~Macdonald,
\textit{``Vector and Geometric Calculus''},
CreateSpace Independent Publishing Platform (2012).

\bibitem{FundamentalTheorem}
G.~Sobczyk and O.~Sánchez,
\textit{``Fundamental Theorem of Calculus''},
\textsf{\doiref{10.1007/s00006-010-0242-8}{Advances~in~Applied~Clifford~Algebras~21,~221~(2011)}},
\href{http://dx.doi.org/10.1007/s00006-010-0242-8}{\texttt{http://dx.doi.org/10.1007/s00006-010-0242-8}}.

\bibitem{Chisolm:2012aa}
E.~Chisolm,
\textit{``Geometric Algebra''},
\texttt{\arxivref{1205.5935}{arxiv:1205.5935}},
\href{http://arxiv.org/abs/1205.5935}{\texttt{http://arxiv.org/abs/1205.5935}}.

\bibitem{macdonaldGC}
A.~Macdonald,
\textit{``{A Survey of Geometric Algebra and Geometric Calculus}''},
\href{http://faculty.luther.edu/~macdonal/GA\&GC.pdf}{\texttt{http://faculty.luther.edu/~macdonal/GA\&GC.pdf}}.

\bibitem{HitzerCalculus}
E.~{Hitzer},
\textit{``{Multivector Differential Calculus}''},
\textsf{ArXiv~e-prints~21,~E.~{Hitzer}~(2013)},
\texttt{\arxivref{1306.2278}{arxiv:1306.2278}}.

\end{thebibliography}

\end{document}